\documentclass{amsart}
\usepackage{amsmath,amssymb,amsthm,amsfonts,amscd,mathrsfs}
\usepackage[all]{xy}
\usepackage{tikz}
\usetikzlibrary{decorations.pathmorphing}
\usepackage{tikz-cd}
 \usepackage{animate}
 \usepackage{graphicx} 
 \usepackage{pb-diagram}
 \usepackage[]{todonotes}
\usepackage{hyperref}

    \newcommand{\LG}{\mathsf{LieGpd}}
    
  \newcommand{\CFG}{\mathsf{CFG}}

\def\Hom{\mathrm{Hom}}
\def\Map{\mathcal{M}ap}
\def\SMap{\mathfrak{Map}}

\newcommand{\cG}{{\mathcal G}}
\newcommand{\cH}{{\mathcal H}}
\newcommand{\cI}{{\mathcal I}}
\newcommand{\cJ}{{\mathcal J}}
\newcommand{\cK}{{\mathcal K}}
\newcommand{\cL}{{\mathcal L}}

\newcommand{\cU}{{\mathcal U}}

\newcommand{\uX}{{\underline X}}

\newcommand{\cat}{\mbox{\rm cat}}

\newcommand{\e}{\epsilon}

\newcommand{\ev}{\operatorname{ev}}

\newcommand{\toto}{\rightrightarrows}

\def\id{\mathrm{id}}

\def\U{{\mathfrak U}}
\def\cU{{\mathcal U}}

\def\S{{\mathfrak S}}
\def\B{{\mathfrak B}}

\def\X{{\mathfrak X}}

\def\Y{{\mathfrak Y}}

\newtheorem{theorem}{Theorem}[section]

\newtheorem{proposition}[theorem]{Proposition}

\theoremstyle{definition}
\newtheorem{definition}[theorem]{Definition}

\newtheorem{example}[theorem]{Example}

\numberwithin{equation}{section}

\begin{document}

\title[The Lusternik-Schnirelmann category for a differentiable stack]
{The Lusternik-Schnirelmann category for a differentiable stack}

\author[Samirah Alsulami]
{Samirah Alsulami}
\address{Department of Mathematics\\
University of Leicester\\
University Road, Leicester LE1 7RH, England, UK}
\email{shba2@le.ac.uk}

\author[Hellen Colman]
{Hellen Colman}
\address{Department of Mathematics\\
Wilbur Wright College\\
4300 N. Narragansett Avenue, Chicago, IL 60634, USA}
\email{hcolman@ccc.edu}

\author[Frank Neumann]
{Frank Neumann}
\address{Department of Mathematics\\
University of Leicester\\
University Road, Leicester LE1 7RH, England, UK}
\email{fn8@le.ac.uk}
\date{\today}  



\begin{abstract}
We introduce the notion of Lusternik-Schnirelmann category for differentiable stacks and establish its relation with the groupoid Lusternik-Schnirelmann category for Lie groupoids. This extends the notion of Lusternik-Schnirelmann category for smooth manifolds and orbifolds. 
\end{abstract}

\maketitle

\section*{Introduction}

The Lusternik-Schnirelmann category or LS-category of a manifold is a numerical invariant introduced by Lusternik and Schnirelmann \cite{LS} in the early 1930Õs as a lower bound on the number of critical points for any smooth function on a compact smooth manifold. Later it was shown that the Lusternik-Schnirelmann category is in fact a homotopy invariant and it became an important tool in algebraic topology and especially homotopy theory. For an overview and survey on the importance of LS-category in topology and geometry we refer the reader to \cite{Ja1, Ja2, CLOT}.

Fundamental in the definition of LS-category of a smooth manifold or topological space is the concept of a categorical set. 
A subset of a space is said to be categorical if it is contractible in the space. The Lusternik-Schnirelmann category $\cat(X)$ of a smooth manifold $X$ is defined to be the least number of categorical open sets required to cover $X$, if that number is finite, otherwise the category $\cat(X)$ is said to be infinite. 

In this article, we generalize the notion of Lusternik-Schnirelmann category to differentiable stacks with the intention of providing a useful tool and invariant to study homotopy theory, the theory of geodesics and Morse theory of differentiable stacks. Differentiable stacks naturally generalize smooth manifolds and orbifolds and are therefore of interest in many areas of geometry, topology and mathematical physics. They are basically generalized smooth spaces where its points also have automorphism groups. For example, they often appear as an adequate replacement of quotients for general Lie group actions on smooth manifolds, especially when the naive quotient does not exist as a smooth manifold. Many moduli and classification problems like for example the classification of Riemann surfaces or vector bundles on Riemann surfaces naturally lead to the notion of a differentiable stack. It can be expected that the stacky LS-category will be a very useful topological invariant for these kind of generalized smooth spaces which appear naturally in geometry and physics. We aim to study the geometrical and topological aspects of the stacky LS-category and its applications in a follow-up article \cite{CN}. Many of these constructions can also be presented in a purely homotopical manner \cite{Al} by employing the homotopy theory of topological stacks \cite{No1, No2}. 

The new notion of stacky LS-category for differentiable stacks presented here employs the notion of a categorical substack and is again an invariant of the homotopy type of the differentiable stack, in fact of the underlying topological stack.  It generalizes the classical LS-category for manifolds \cite{LS} and we show that it is directly related with the groupoid LS-category for Lie groupoids as defined by Colman \cite{Co1} for Lie groupoids.

The material of this article is organised as follows: In the first section we collect the basic definitions of differentiable stacks and Lie groupoids and establish some fundamental properties. In particular we exhibit the various connections between differentiable stacks and Lie groupoids. The second section recalls the foundations of Lusternik-Schnirelmann category for groupoids and its Morita invariance. In the third section we introduce the new notion of stacky Lusternik-Schnirelmann category and establish its relationship with the groupoid LS-category of the various groupoids introduced in the first section.

\section{Differentiable stacks and Lie groupoids}
In this section we will collect in detail the notions and some of the fundamental properties of
differentiable stacks and Lie groupoids, which we will use later.  
We refer the reader to various resources on differentiable stacks \cite{B1, B2, BX, BN, Fa, H} and 
on Lie groupoids \cite{LTX, MM, TXL, FeN} for more details and specific examples and their interplay.

Differentiable stacks are defined over the category of smooth
manifolds. A smooth manifold here will always mean a finite
dimensional second countable smooth manifold, which need not
necessarily be Hausdorff. We denote the category of smooth manifolds
and smooth maps by $\mathfrak S$.

A {\it submersion} is a smooth map $f: U\rightarrow X$ such that the
derivative $f_{*}: T_uU\rightarrow T_{f(u)}X$ is surjective for all
points $u\in U$. The dimension of the kernel of the linear map $f_*$
is a locally constant function on $U$ and called the {\it relative
dimension} of the submersion $f$. An {\it \'etale morphism} is a
submersion of relative dimension $0$. This means that a morphism $f$
between smooth manifolds is {\it \'etale} if and only if $f$ is a
local diffeomorphism.

The {\it \'etale site} ${\mathfrak S}_{et}$ on the category
$\mathfrak S$ is given by the following Grothendieck
topology on $\mathfrak S$. We call a family $\{U_i\rightarrow X\}$
of morphisms in $\mathfrak S$ with target $X$ a {\it covering
family} of $X$, if all smooth maps $U_i\rightarrow X$ are \'etale
and the total map $\coprod_i U_i\rightarrow X$ is surjective. This
defines a pretopology on $\mathfrak S$ generating a Grothendieck
topology, the {\it \'etale topology} on $\mathfrak S$ (see
\cite{SGA4}, Expos\'e II).

As remarked in \cite{BX}, not all fibre products for two morphisms
$U\rightarrow X$ and $V\rightarrow X$ exist in $\mathfrak S$, but if
at least one of the two morphisms is a submersion, then the fibre
product $U\times_X V$ exists in $\mathfrak S$, which will be enough,
while dealing with differentiable stacks.

\begin{definition}
A {\it category fibred in groupoids} over $\mathfrak S$ is a
category $\mathfrak X$, together with a functor $\pi_{\mathfrak X}:
\mathfrak X\rightarrow \mathfrak S$ such that the following axioms
hold:
\begin{itemize}
\item[(i)] For every morphism $V\rightarrow U$ in $\mathfrak S$, and
every object $x$ of $\mathfrak X$ lying over $U$, there exists an
arrow $y\rightarrow x$ in $\mathfrak X$ lying over $V\rightarrow U$,
\item[(ii)] For every commutative triangle $W\rightarrow V\rightarrow U$
in $\mathfrak S$ and morphisms $z\rightarrow x$ lying over
$W\rightarrow U$ and $y\rightarrow x$ lying over $V\rightarrow U$,
there exists a unique arrow $z\rightarrow y$ lying over
$W\rightarrow V$ such that the composition $z\rightarrow
y\rightarrow x$ is the morphism $z\rightarrow x$.
\end{itemize}
\end{definition}

The axiom (2) ensures that the object $y$ over $V$, which exists
after (1) is unique up to a unique isomorphism. Any choice of such
an object $y$ is called a {\it pullback} of $x$ via the morphism $f:
V\rightarrow U$. We will write as usual $y=x|V$ or $y=f^*x$.

Let $\mathfrak X$ be a category fibred in groupoids over $\mathfrak
S$. Occasionally we will denote by $X_0$ the collection of objects and $X_1$ 
the collection of arrows of the category $\mathfrak X$. 
The subcategory of $\mathfrak X$ of all objects lying over a
fixed object $U$ of $\mathfrak S$ with morphisms being those lying
over the identity morphism $id_U$ is called the {\it fibre} or {\it
category of sections} of $\mathfrak X$ over $U$, which will be
denoted by ${\mathfrak X}_U$ or ${\mathfrak X}(U)$. By definition
all fibres ${\mathfrak X}_U$ are discrete groupoids.

Categories fibred in groupoids form a 2-category, denoted by $\CFG$. The $1$-morphisms are given by functors $\phi:
{\mathfrak X} \rightarrow {\mathfrak Y}$ respecting the projection
functors i. e. $\pi_{\mathfrak Y}\circ \phi = \pi_{\mathfrak X}$ and the
$2$-morphisms are given by natural transformations between these
functors preserving projection functors. Fibre products
exist in $\CFG$ (see \cite{Gr}).

We will say that the categories fibred in groupoids $\X$ and $\Y$ are {\em isomorphic} if there are $1$-morphisms $\phi:\X \rightarrow \Y$ and $\psi:\Y \rightarrow \X$ and  
 $2$-morphisms $T$ and $T'$ such that $T:\phi\circ \psi \Rightarrow \id_{\Y}$ and $T':\psi\circ \phi \Rightarrow \id_{\X}$.

\begin{example}[Identity]
Let $\X$ be the fixed category $\S$. Let $\pi=\id_{\S}: \S \to \S$ be the projection functor. Then $\X= \S$ together with the identity map is a category fibred in groupoids.
\end{example}

\begin{example}[Object]
Given a fixed object $X\in \S$, i.e. a smooth manifold, consider the category $\uX$ whose objects are $(U, f)$ where $f: U \to X$ is a morphism in $\S$ and $U$ an object in $\S$, and whose arrows are diagrams 
\[
\xymatrix{ U\ar[rr]^\phi\ar[rd]_{f}&&V \ar[ld]^{g}\\
&X& }
\]
The projection functor is $\pi:\uX \to \S$ with $\pi((U, f))=U$ and $\pi((U, f),\phi, (V, g))=\phi$. We have that $\uX$ is a category fibred in groupoids.

In particular, in case that $X$ is a point, $X=*$, we have that $\underline *= \S$.
\end{example}

\begin{example}[Sheaves]
Let $F: {\mathfrak S}\rightarrow (Sets)$ be a presheaf, i. e. a
contravariant functor. We get a category fibred in groupoids
${\mathfrak X}$, where the objects are pairs $(U, x)$, with  $U$ a
smooth manifold and $x\in F(U)$ and a morphism $(U, x) \rightarrow
(V, y)$ is a smooth map $f: U \rightarrow V$ such that $x=y|_F U$,
i. e. $x=F(f)(y)$. The projection functor is given by $$\pi:
{\mathfrak X} \rightarrow {\mathfrak S},\,\, (U, x)\mapsto U.$$

Especially any sheaf $F: {\mathfrak S}\rightarrow (Sets)$ gives
therefore a category fibred in groupoids over $\mathfrak S$ and in
particular we see again that every smooth manifold $X$ gives a category fibred in
groupoids $\underline X$ over $\mathfrak S$ as the sheaf represented
by $X$, i. e. where
$${\underline X}(U)= Hom_{\mathfrak S}(U, X).$$
To simplify notation, we will sometimes freely identify $\underline X$ with
the smooth manifold $X$.
\end{example}

Now let us recall the definition of a stack \cite{BX}. In the
following let $\mathfrak S$ always be the category of smooth
manifolds equipped with the \'etale topology as defined above. We
could of course replace the \'etale topology with any other
Grothendieck topology on $\mathfrak S$, but in this article we are
mainly interested in stacks over the \'etale site ${\mathfrak
S}_{et}$.

\begin{definition} A category fibred in groupoids
$\mathfrak X$ over $\mathfrak S$ is a {\it stack} over $\mathfrak S$
if the following gluing axioms hold:
\begin{itemize}
\item[(i)] For any smooth manifold $X$ in $\mathfrak S$, any two objects
$x, y$ in ${\mathfrak X}$ lying over $X$ and any two isomorphisms
$\phi, \psi: x\rightarrow y$ over $X$, such that $\phi|U_i=\psi|U_i$
for all $U_i$ in a covering $\{U_i\rightarrow X\}$ it follows that
$\phi=\psi$.
\item[(ii)] For any smooth manifold $X$ in $\mathfrak S$, any two objects
$x, y \in {\mathfrak X}$ lying over $X$, any covering
$\{U_i\rightarrow X\}$ and, for every $i$, an isomorphism $\phi_i:
x|U_i\rightarrow y|U_i$, such that $\phi|U_{ij}=\phi_j|U_{i j}$ for
all $i,j$, there exists an isomorphism $\phi: x\rightarrow y$ with
$\phi|U_i=\phi$ for all $i$.
\item[(iii)] For any smooth manifold $X$ in $\mathfrak S$, any covering
$\{U_i\rightarrow X\}$, any family $\{x_i\}$ of objects $x_i$ in the
fibre ${\mathfrak X}_{U_i}$ and any family of morphisms
$\{\phi_{ij}\}$, where $\phi_{ij}: x_i|U_{ij} \rightarrow
x_j|U_{ij}$ satisfying the cocycle condition $\phi_{jk}\circ
\phi_{ij}=\phi_{ik}$ in ${\mathfrak X}(U_{ijk})$ there exist an
object $x$ lying over $X$ with isomorphisms $\phi_i:
x|U_i\rightarrow x_i$ such that $\phi_{ij}\circ \phi_i=\phi_j$ in
${\mathfrak X}(U_{ij})$.
\end{itemize}
\end{definition}

The isomorphism $\phi$ in (ii) is unique by (i) and similar from (i)
and (ii) it follows that the object $x$ whose existence is asserted
in (iii) is unique up to a unique isomorphism. All pullbacks
mentioned in the definitions are also only unique up to isomorphism,
but the properties do not depend on choices.

In order to do geometry on stacks, we have to compare them with
smooth manifolds.

A category $\mathfrak X$ fibred in groupoids over $\mathfrak S$ is
{\it representable} if there exists a smooth manifold $X$ such that
$\uX$ is isomorphic to $\mathfrak X$ as categories fibred in groupoids
over $\mathfrak S$.

A morphism of categories fibred in
groupoids $\mathfrak X \rightarrow \mathfrak Y$ is {\it representable}, if for every smooth manifold $U$ and every morphism
$\underline U\rightarrow {\mathfrak Y}$ the fibred product ${\mathfrak
X}\times_{\mathfrak Y} \underline U$ is representable.

A morphism of categories fibred in
groupoids $\mathfrak X \rightarrow \mathfrak Y$ is a {\it representable
submersion}, if it is representable and the induced morphism
of smooth manifolds ${\mathfrak X}\times_{\mathfrak Y} U \rightarrow
U$ is a submersion for every smooth manifold $U$ and every morphism
$\underline U\rightarrow {\mathfrak Y}$.

\begin{definition}
A stack $\mathfrak X$ over $\mathfrak S$ is a {\it differentiable}
or {\it smooth stack} if there exists a smooth manifold $X$ and a
surjective representable submersion $x: X\rightarrow {\mathfrak X}$,
i. e. there exists a smooth manifold $X$ together with a morphism of
stacks $x: X\rightarrow {\mathfrak X}$ such that for every smooth
manifold $U$ and every morphism of stacks $U\rightarrow {\mathfrak
X}$ the fibre product $X\times_{\mathfrak X} U$ is representable and
the induced morphism of smooth manifolds $X\times_{\mathfrak X} U$
is a surjective submersion.
\end{definition}

If $\mathfrak X$ is a differentiable stack, a surjective representable submersion $x:
X \rightarrow {\mathfrak X}$ as before is called a {\it presentation of}
$\mathfrak X$ or {\it atlas for} $\mathfrak X$. It need not be
unique, i. e. a differentiable stack can have different
presentations.

\begin{example} All representable stacks are differentiable stacks.
Let $X$ be a smooth manifold. The category fibred in groupoids
$\underline X$ is in fact a differentiable stack over $\mathfrak S$ since it is representable.
A presentation is given by the identity morphism $\id_X$, which is in fact a diffeomorphism.
\end{example}

\begin{example}[Torsor]\label{BG}
Let $G$ be a Lie group.
Consider the category  $\B G$ which has as objects principal $G$-bundles (or $G$-torsors) $P$ over $S$ and as arrows commutative diagrams
\[
\xymatrix{ P\ar[r]^\psi\ar[d]_\pi&Q\ar[d]^\tau\\
S\ar[r]^{\varphi}& T}
\]
where the map $\psi: P \to Q$ is equivariant.

The category $\B G$ together with the projection functor $\pi: \B G \to \S$  given by $\pi(P\to S)= S$ and $\pi((\psi, \varphi))= \varphi$ is a category fibred in groupoids, in fact a differentiable stack, the {\it classifying stack} of $G$ whose atlas presentation is given by the representable surjective submersion
$*\rightarrow {\mathfrak B}G$.

\end{example}

\begin{example}[Quotient Stack]\label{Quotient}
Let $X$ be a smooth manifold with a smooth (left) action $\rho: G\times X\rightarrow X$ by a Lie group $G$.
Let $[X/G]$ be the category which has as objects triples $(P, S, \mu)$, where $S$ is a smooth manifold of
$\mathfrak S$, $P$ a principal $G$-bundle (or $G$-torsor) over $S$ and $\mu: P\rightarrow X$ a $G$-equivariant smooth map.
A morphism $(P, S, \mu)\rightarrow (Q, T, \nu)$ is a commutative diagram

\[
\xy \xymatrix{ & X &\\ P\ar[ur]^\mu\ar[rr]^{\psi}\ar[d]_\pi& &
Q\ar[ul]_\nu\ar[d]^\tau\\
S\ar[rr]^{\varphi}&& T}
\endxy
\]

\noindent where $\psi: P\rightarrow Q$ is a $G$-equivariant map.
Then $[X/G]$ together with the projection functor $\pi:
[X/G]\rightarrow {\mathfrak S}$ given by $\pi( (P, S, \mu)=S$ and
$\pi((\psi, \varphi))= \varphi$ is a category fibred in groupoids over $\mathfrak S$. $[X/G]$ is
in fact a differentiable stack, the {\it quotient stack} of $G$.
An atlas is given by the representable surjective submersion
$x: X\rightarrow [X/G]$.

If $X=*$ is just a point, we simply recover the differentiable stack ${\mathfrak B}G$ as defined in the previous example, i.e. $[*/G]=\B G$.

In some way, quotient stacks encode in a non-equivariant and systematic way various equivariant data of general Lie group actions, which need not to be free. 
\end{example}

Differentiable stacks are basically incarnations of Lie groupoids.

\begin{definition}
A {\it Lie groupoid} $\mathcal G$ is a groupoid in the category
$\mathfrak S$ of smooth manifolds, i. e.
$${G_1\rightrightarrows G_0}$$
such that the space of arrows $G_1$ and the space of objects $G_0$
are smooth manifolds and all structure morphisms
$$m: G_1\times_{G_0} G_1\rightarrow G_1, \, s,t: G_1\rightarrow
G_0,\, i: G_1\rightarrow G_0, \, e: G_0\rightarrow G_1$$ are smooth
maps and additionally the source map $s$ and the target map $t$ are
submersions.
\end{definition}
Morphisms between Lie groupoids are given by functors $(\phi_1, \phi_0)$ where $\phi_i: G_i\to G'_i$ is a smooth map. We will call them {\em strict morphisms}.

Lie groupoids, strict morphisms and natural transformations form a 2-category that we denote $\LG$.

We will show now a construction of a Lie groupoid associated to a differentiable stack.
Let $\mathfrak X$ be a differentiable stack with a given
presentation $x: X\rightarrow {\mathfrak X}$. We can associate to
($\mathfrak X, x)$ a Lie groupoid ${\mathcal G}(x)=G_1\rightrightarrows
G_0$ as follows: Let $G_0=X$ and $G_1=X\times_{\mathfrak X} X$. The source and
target morphisms $s, t: X\times_{\mathfrak X} X\rightrightarrows X$
of $\mathcal G$ are given as the two canonical projection morphisms.
The composition of morphisms $m$ in $\mathcal G$ is given as
projection to the first and third factor
$$X\times_{\mathfrak X} X\times_{\X} X \cong (X\times_{\mathfrak X} X)\times_X
(X\times_{\mathfrak X} X)\rightarrow X\times_{\mathfrak X} X.$$ The
morphism which interchanges factors $X\times_{\X} X\rightarrow
X\times_{\X} X$ gives the inverse morphism $i$ and the unit morphism
$e$ is given by the diagonal morphism $X\rightarrow X\times_{\X} X$.
Because the presentation $x: X\rightarrow {\mathfrak X}$ of a
differentiable stack is a submersion, it follows that the source and
target morphisms $s, t: X\times_{\mathfrak X} X\rightrightarrows X$
are submersions as induced maps of the fibre product.

The Lie groupoid associated to the differentiable stack $\mathfrak X$ in this way is also part of a simplicial smooth manifold $X_{\bullet}$, whose homotopy type encodes the homotopy type of $\mathfrak X$ \cite{No2, FeN, S}.

Given instead a Lie groupoid we can associate a differentiable stack to it. Basically this is a generalization of Example \ref{BG} where we associate to a Lie group $G$ the classifying stack ${\mathfrak B}G$.

Let's introduce first the notions of {\em groupoid action} and {\em groupoid torsor}.

\begin{definition}[Action of a groupoid on a manifold]
Let ${\cG}$ be a Lie groupoid $G_1\rightrightarrows G_0$ and
$P$ a manifold in $\S$. 
An {\em action of $\cG$ on  $P$} is given by 
\begin{enumerate}
\item[(i)] an anchor map $a: P \to G_0$
\item[(ii)] an action map $\mu: G_1 \times_{s,a} P \to P$
\end{enumerate}
such that $t(k)=a(k\cdot p)$ for all $(k, p)\in G_1 \times P$ with $s(k)=a(p)$, satisfying the standard action properties: $e(a(p)) \cdot p = p$ and $(k \cdot p)\cdot h=(gh)\cdot p$ whenever the operations are defined.
\end{definition}

\begin{definition}[Groupoid torsor]
Let ${\cG}$ be a Lie groupoid $G_1\rightrightarrows G_0$ and
$S$ a manifold in $\S$.  A {\em $\cG$-torsor over $S$} is given by 

\begin{enumerate}
\item[(i)] a manifold $P$ together with 
\item[(ii)] a surjective submersion $\pi: P\rightarrow S$ and
\item[(iii)] an action of $\cG$ on  $P$ 
\end{enumerate}
such that for all $p, p'\in P$ with $\pi(p)=\pi(p')$ there exists a unique $k\in
G_1$ such that $k\cdot p=p'$ 

\end{definition}

Let $\pi:
P\rightarrow S$ and $\rho: Q\rightarrow T$ be $\cG$-torsors. A {\it
morphism} of $\cG$-torsors is given by a commutative diagram 
\[
\xymatrix{ P\ar[r]^\psi\ar[d]_\pi&Q\ar[d]^\tau\\
S\ar[r]^{\varphi}& T}
\]
such that $\psi$ is a $\cG$-equivariant map.

\begin{example}[$\cG$-torsors]\label{BGr}
Let ${\mathcal G}$ be a fixed Lie groupoid $G_1\rightrightarrows G_0$. Consider the category  $\B\cG$ which has as objects $\cG$-torsors $P$ over $S$ and as  arrows morphisms of $\cG$-torsors as described above.

Then $\B\cG$ is a category fibred in groupoids over $\mathfrak
S$ with  projection functor $\pi: \B\cG   \rightarrow {\S}$
given by $\pi(P\to S)= S$ and $\pi((\psi, \varphi))= \varphi$. Moreover, $\B\cG$ is a differentiable stack.
 \end{example}

We have the following fundamental property (see for example \cite[Prop. 2.3]{BX}) of $\B\cG$.
\begin{theorem}
For every Lie groupoid ${\mathcal G}={G_1\rightrightarrows G_0}$ the category
fibred in groupoids ${\mathfrak B}\cG$ of $\cG$-torsors is a differentiable stack
with a presentation $\tau_0: G_0\rightarrow {\mathfrak B}\cG$. 
\end{theorem}

The stack ${\mathfrak B}\cG$ is also called the {\it classifying stack} of $\cG$-torsors.
It follows from this also that the Lie groupoid ${\mathcal G}={G_1\rightrightarrows G_0}$
is isomorphic to the Lie groupoid ${\mathcal G}(\tau_0)$ associated to the atlas 
$\tau_0: G_0\rightarrow {\mathfrak B}\cG$ of the stack ${\mathfrak B}\cG$.

Under this correspondence between differentiable stacks and Lie groupoids, the quotient stack $[X/G]$ of an action of a Lie group $G$ on a smooth manifold $X$ as described in Example \ref{Quotient} corresponds to the action groupoid $G\times X\rightrightarrows X$ \cite{BX, H}.  In particular, if $X$ is just a point the classifying stack ${\mathfrak B}G$ of a Lie group $G$ corresponds to the Lie groupoid $G\rightrightarrows *$. 

As the presentations of a differentiable stack are not unique, the
associated Lie groupoids might be different. In order to define
algebraic invariants, like cohomology or homotopy groups for
differentiable stacks they should however not depend on a chosen
presentation of the stack. Therefore it is important to know, when
two different Lie groupoids give rise to isomorphic stacks. This
will be the case when the Lie groupoids are Morita equivalent.

\begin{definition}
Let ${\mathcal G}={G_1\rightrightarrows G_0}$ and ${\mathcal
H}=H_1\rightrightarrows H_0$ be Lie groupoids. A morphism of Lie
groupoids is a smooth functor $\phi: {\mathcal G}\rightarrow
{\mathcal H}$ given by two smooth maps $\phi=(\phi_1, \phi_0)$ with
$$\phi_0: G_0\rightarrow H_0, \, \phi_1: G_1\rightarrow H_1$$
which commute with all structure morphisms of the groupoids. A
morphism $\phi: {\mathcal G}\rightarrow {\mathcal H}$ of Lie
groupoids is a {\it Morita morphism} or {\it essential equivalence}
if
\begin{itemize}
\item[(i)] $\phi_0: G_0\rightarrow H_0$ is a surjective submersion,
\item[(ii)] the diagram
\[
\xy \xymatrix{ G_1\ar[r]^{(s,t)} \ar[d]_{\phi_1}&
G_0\times G_0\ar[d]^{\phi_0 \times \phi_0}\\
H_1\ar[r]^{(s,t)}& H_0\times H_0}
\endxy
\]
is cartesian, i. e. $G_1\cong H_1\times_{H_0\times H_0} G_0\times
G_0$.
\end{itemize}
Two Lie groupoids $\mathcal G$ and $\mathcal H$ are {\it Morita
equivalent}, if there exists a third Lie groupoid $\mathcal K$ and
Morita morphisms
$${\mathcal G}\stackrel{\phi}\leftarrow {\mathcal K}\stackrel{\psi}\rightarrow {\mathcal
H}$$
\end{definition}

We have the following main theorem concerning the relation of the
various Lie groupoids associated to various presentations of a
differentiable stack \cite{BX}, Theorem 2.24.

\begin{theorem}\label{MIst}
Let ${\mathcal G}$ and ${\mathcal
H}$ be Lie groupoids. Let $\mathfrak X$
and $\mathfrak Y$ be the associated differentiable stacks, i. e.
$\X=\B\cG$ and
$\Y=\B\cH$. Then the
following are equivalent:
\begin{itemize}
\item[(i)] the differentiable stacks $\mathfrak X$ and $\mathfrak Y$ are
isomorphic,
\item[(ii)] the Lie groupoids $\mathcal G$ and $\mathcal H$ are
Morita equivalent.
\end{itemize}
\end{theorem}

As a special case we have the following fundamental property concerning different presentations of a differentiable stack $\mathfrak X$.

\begin{proposition}\label{MIstpres} Let $\mathfrak X$ be a differentiable stack with two given presentations $x: X\rightarrow {\mathfrak X}$ and $x': X'\rightarrow {\mathfrak X}$. Then the associated Lie groupoids $\cG(x)$ and $\cG(x')$ are Morita equivalent.
\end{proposition}

Therefore Lie groupoids present isomorphic differentiable stacks if and only if
they are Morita equivalent or in other words differentiable stacks
correspond to Morita equivalence classes of Lie groupoids.

We now recall the fundamental notion of a smooth morphism between differentiable stacks (see for example \cite{H}).

\begin{definition}[Smooth morphism]
An arbitrary morphism $\X \to \Y$ of differentiable stacks is {\em smooth}, if there are atlases $X\to \X$ and $Y\to \Y$ such that the induced morphism from the fibered product $X \times_{\Y} Y\to Y$ in the diagram below is a smooth map between manifolds.
\[
\xy \xymatrix{ X \times_{\Y} Y\ar[r]^{} \ar[dd]_{}&
X\ar[d]^{}\\
&\X\ar[d]^{}\\
Y\ar[r]^{}& \Y}
\endxy
\]
\end{definition}

Let $\U$ be a subcategory of $\X$. Recall that a subcategory is called {\em saturated} if whenever it contains an object $x$ then it contains the entire isomorphism class $\bar x$ of that object and is called {\em full} if whenever it contains an arrow between $x$ and $y$, it contains the entire set $\Hom(x,y)$ of arrows.

Let $\pi: \X\to \S$ be a differentiable stack and $x: X\to \X$ be an atlas. Let $U\subset X$ be an open subset and consider the saturation $U_0$ of the image $x(U)$ in $X_0$, i.e.
$$U_0=\{z\in X_0 |\; z\in \bar x \mbox{ for some } x\in U\}.$$
The full subcategory $\U$ on $U_0$ is $U_1\toto U_0$ where $U_1=\{g\in X_1 |\; s(g), t(g)\in U_0\}$.

\begin{definition}[Restricted substack] Let $\pi: \X\to \S$ be a differentiable stack with atlas $x: X\to \X$
and $U\subset X$ be an open set. Consider the full subcategory $\U$ on $U_0$ and let $\pi':=\pi\circ i$, where $i: \U\to \X$ is the inclusion. We say that $\U$ with the projection $\pi': \U\to \S$ is the {\em restricted substack} of $\X$ to $\U$.
\end{definition}

\begin{definition}[Constant morphism] Let $c:\X \to \Y$ be a smooth morphism between differentiable stacks. We say that $c$ is a {\em constant morphism} if there are presentations $X\to \X$ and $Y\to \Y$ such that the induced morphism from the fibered product $X \times_{\Y} Y\to Y$ is a constant map.
\end{definition}

For instance, any smooth morphism  $\X \to \Y$ where $\Y$ admits a presentation by a point $*\to \Y$ is a constant morphism.

\begin{example}\label{circle}
Let $S^1$ act on $S^1$ by rotation and consider the quotient stack $\X$ associated to this action, $\X=[S^1/{S^1}]$. We will show that the identity map $\id_{[S^1/{S^1}]}:[S^1/{S^1}]\to [S^1/{S^1}]$ is a constant map. The groupoid $S^1\times {S^1}\rightrightarrows S^1$ is Morita equivalent to a point groupoid $* \rightrightarrows*$, therefore the stacks $\X=[S^1/{S^1}]$ and $\underline *$ are isomorphic. Since $*\to \underline *$ is a presentation for $\underline *$ it follows that $*\to [S^1/{S^1}]$ is a presentation for $\X$. Hence any map with codomain $\X=[S^1/{S^1}]$ is a constant morphism of stacks.
\end{example}

Let us finish this section by remarking that it is also possible to define stacks and groupoids over the more general category of diffeological spaces instead of using the category of smooth manifolds as we have done here \cite{So, I-Z, CLW}. 

\section{Lusternik-Schnirelmann category for Lie groupoids}

We will first recall the definition and fundamental properties for the notion of Lusternik-Schnirelmann category for Lie groupoids.
The most important property here is that the Lusternik-Schnirelmann category of a Lie groupoid is in fact Morita invariant, 
which means that it is in fact an invariant of the associated differentiable stack. We will follow here the general approach of \cite{Co1}, where
the notion of Lusternik-Schnirelmann category for Lie groupoids was first introduced.

Our context will be the Morita bicategory of Lie groupoids $\LG(E^{-1})$ obtained from $\LG$ by formally inverting the essential equivalences $E$. Objects in this bicategory are Lie groupoids,  $1$-morphisms are {\em generalized maps}
\[\cK\overset{\e}{\gets}\cJ\overset{\phi}{\to}\cG\]
such that $\e$ is an essential equivalence and 
2-morphisms from $\cK\overset{\e}{\gets}\cJ\overset{\phi}{\to}\cG$ to $\cK\overset{\e'}{\gets}\cJ'\overset{\phi'}{\to}\cG$ are given by classes of diagrams:
$$
\xymatrix{ &
{\cJ}\ar[dr]^{\phi}="0" \ar[dl]_{\e }="2"&\\
{\cK}&{\cL} \ar[u]_{u} \ar[d]^{v}&{\cG}\\
&{\cJ'}\ar[ul]^{\e'}="3" \ar[ur]_{\phi'}="1"&
\ar@{}"0";"1"|(.4){\,}="7"
\ar@{}"0";"1"|(.6){\,}="8"
\ar@{}"7" ;"8"_{\sim}
\ar@{}"2";"3"|(.4){\,}="5"
\ar@{}"2";"3"|(.6){\,}="6"
\ar@{}"5" ;"6"^{\sim}
}
$$
where $\cL$ is a topological groupoid, and $u$ and $v$ are essential  equivalences.

The {\em path groupoid} of $\cG$ is defined as the mapping groupoid  in this bicategory, $P\cG=\Map(\cI, \cG)$.

Let $(\sigma, f): \cK\overset{\sigma}{\gets}\cK'\overset{f}{\to} \cG$ and  $(\tau, g): \cK\overset{\tau}{\gets}\cK''\overset{g}{\to} \cG$ be generalized maps.
The map $(\sigma, f)$ is {\em groupoid homotopic} to $(\tau, g)$ if there exists $(\e, H):  \cK\overset{\e}{\gets}\tilde \cK\overset{H}{\to} P\cG$ and two commutative diagrams up to natural transformations:

$$
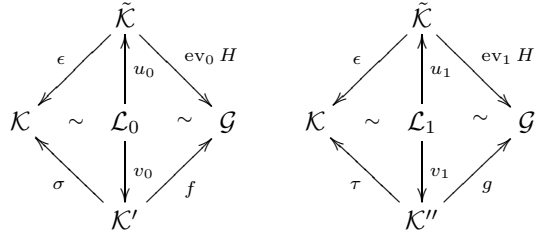
\xymatrix{ &
{\tilde \cK}\ar[dr]^{\ev_{0}H}="0" \ar[dl]_{\e }="2"&\\
{\cK}&{\cL_0} \ar[u]_{u_0} \ar[d]^{v_0}&{\cG}\\
&{\cK'}\ar[ul]^{\sigma}="3" \ar[ur]_{f}="1"&
\ar@{}"0";"1"|(.4){\,}="7"
\ar@{}"0";"1"|(.6){\,}="8"
\ar@{}"7" ;"8"_{\sim}
\ar@{}"2";"3"|(.4){\,}="5"
\ar@{}"2";"3"|(.6){\,}="6"
\ar@{}"5" ;"6"^{\sim}
}\;\;\;\;\;\;\;
\xymatrix{ &
{\tilde \cK}\ar[dr]^{\ev_{1}H}="0" \ar[dl]_{\e }="2"&\\
{\cK}&{\cL_1} \ar[u]_{u_1} \ar[d]^{v_1}&{\cG}\\
&{\cK''}\ar[ul]^{\tau}="3" \ar[ur]_{g}="1"&
\ar@{}"0";"1"|(.4){\,}="7"
\ar@{}"0";"1"|(.6){\,}="8"
\ar@{}"7" ;"8"_{\sim_{}}
\ar@{}"2";"3"|(.4){\,}="5"
\ar@{}"2";"3"|(.6){\,}="6"
\ar@{}"5" ;"6"^{\sim}
}
$$
where $\cL_i$ is an action groupoid, and $u_i$ and $v_i$ are equivariant essential  equivalences for $i=0,1$.

Similarly as for differentiable stacks we also have the concept of a restricted groupoid over a given invariant subset and that of a generalized constant map, which we will need to define LS-category for Lie groupoids. 

\begin{definition}
Let ${\mathcal G}=G_1\rightrightarrows G_0$ be a Lie groupoid. An open set $U\subset G_0$ is {\it invariant} if $t(s^{-1}(U))\subset U$.  The {\it restricted groupoid $\cU$} to an invariant set $U\subset G_0$ is the full groupoid over $U$. In other words, $U_0=U$ and $U_1=\{g\in G_1: s(g), t(g) \in U\}$. We write $\cU=\cG|_U\subset \cG$.
\end{definition}

\begin{definition} We say that the map $(\e, c): \cK\overset{\e}{\gets}\cK'\overset{c}{\to} \cG$ is a {\em generalized constant map} if for all $x\in K'_0$ there exists $g\in G_1$ with $s(g)=x_0$ such that $c(x)=t(g)$ for a fixed $x_0\in G_0$.
\end{definition}
In other words, the image of the generalized map $(\e, c)$ is contained in a fixed orbit $\mathcal{O}$, the orbit of $x_0$.

\begin{definition} For an invariant open set $U\subset G_0$, we will say that the restricted groupoid $\cU$ is $\cG$-{\em categorical} if the inclusion map $i_{\cU}\colon \cU\to \cG$ is groupoid homotopic to a generalized constant map $(\e, c)$.
\end{definition}

In other words,  the diagram $$\xymatrix{\cL\ar[r]^-{c}\ar[d]_{\e}&\mathcal{O}\ar[d]^{}\\ \cU\ar[r]^{}&\cG}$$ is commutative up to groupoid homotopy 
where $\e$ is an equivariant essential equivalence and $\mathcal{O}$ an orbit. 

Now we can make the following definition (see \cite{Co2}).

\begin{definition} Let $\cG= G_1\rightrightarrows G_0$ be a Lie groupoid. The {\em groupoid Lusternik-Schnirelmann} or {\em groupoid LS-category}, $\cat(\cG)$, is the least number of  invariant open sets $U$ needed to cover $G_0$ such that  the restricted groupoid $\cU$ is  $\cG$-categorical. 

If $G_0$ cannot be covered by a finite number of such open sets, we will say that  $\cat(\cG)=\infty$.
\end{definition}

We have the following important property of the groupoid Lusternik-Schnirelmann category (see \cite{Co1}).

\begin{theorem} \label{MI}
The Lusternik-Schnirelmann category of a Lie groupoid is invariant under Morita equivalence of Lie groupoids, i.e. if $\cG$ is a Lie groupoid which is Morita equivalent to a Lie groupoid $\cG'$, then we have
$$\cat(\cG)=\cat(\cG').$$
\end{theorem}

The groupoid Lusternik-Schnirelmann category also generalizes the ordinary Lusternik-Schnirelmann category of a smooth manifold. In fact,  if $\cG=u(M)$ is the unit groupoid, then we have $\cat(\cG)=\cat(M)$, where $\cat$ on the right hand side means the ordinary Lusternik-Schnirelmann category of a smooth manifold.

\section{Lusternik-Schnirelmann category of a differentiable stack}
Using homotopical properties of differentiable stacks we will now introduce the Lusternik-Schnirelmann category of a differentiable stacks. Let $\X$ be a differentiable stack. Consider the path stack $P\X = \SMap([0, 1], \X)$ of $\X$ as defined by Noohi in \cite{No1} for general topological stacks. We will say that the morphisms $f:\X\to\Y$ and $g:\X\to\Y$ between differentiable stacks are {\em homotopic} if there exists a morphism of stacks $H:\X\to P\Y$ such that the following diagram of stack morphisms is commutative up to natural transformations:

\[
\xymatrix{ \Y& P\Y \ar[r]^{\ev_1}\ar[l]_{\ev_0}&\Y\\ 
&\X\ar[lu]^{f}\ar[ru]_{g}\ar[u]^{H}&}
\]

\begin{definition} Let $\pi: \X\to \S$ be a differentiable stack with atlas $x: X\to \X$
and $U\subset X$ be an open set. We will say that the restricted substack $\U$ is $\X$-{\em categorical} if the inclusion map $i_{\U}\colon \U\to \X$ is homotopic to a constant morphism $c:\U \to \X$ between differentiable stacks.
\end{definition}

For instance, in Example \ref{circle} for the stack $\X=[S^1/S^1]$ let $U$ be the set of triples $(P, S, \mu)$, where $S$ is a smooth manifold, $P$ a $S^1$-torsor over $S$ and $\mu: P\rightarrow S^1$ a $S^1$-equivariant smooth map. That is, $\U=[S^1/S^1]$. We have that the stack $\U$ is $\X$-categorical since the identity map $\id: [S^1/S^1]\to [S^1/S^1]$ is homotopic to a constant morphism of stacks.

\begin{definition} Let $\pi: \X\to \S$ be a differentiable stack with atlas $x: X\to \X$. The {\em stacky Lusternik-Schnirelmann} or {\em stacky LS-category}, $\cat(\X)$, is the least number of open sets $U$ needed to 
cover $X$ such that  the restricted substack $\cU$ is $\X$-categorical. 

If $X$ cannot be covered by a finite number of such open sets, we will say that  $\cat(\X)=\infty$.
\end{definition}

\begin{example} 
Let $X$ be a smooth manifold.  It follows immediately from the above definition that the stacky LS-category $\cat({\underline{X}})$ is equal to the classical LS-category $\cat(X)$ of the manifold. 
\end{example} 

\begin{example}
From Example \ref{circle}, we get immediately for the LS-category of the quotient stack $[S^1/S^1]$ that $\cat([S^1/S^1])=1$.
\end{example}

The following theorems establish the relationship between Lusternik-Schnirelmann category of differentiable stacks and Lie groupoids (see also \cite{CN}).

\begin{theorem}
Let $\X$ be a differentiable stack with a given presentation $x: X_0\rightarrow {\X}$ and associated Lie groupoid $\cG(x)=X_0\times_{\X} X_0\toto X_0$. Then
$$\cat(\X)=\cat(\cG(x)).$$
\end{theorem}

\begin{proof} This follows from the definition of LS-category for Lie groupoids. The fact that LS-category of Lie groupoids is Morita invariant, implies now by using Proposition \ref{MIstpres} that it does not depend on the chosen presentation for the differentiable stack $\X$, which gives the result.
\end{proof}

\begin{theorem}
Let $\cG$ be a Lie groupoid and $\B\cG$ be the associated differentiable stack. Then 
$$\cat(\B\cG)=\cat(\cG).$$
\end{theorem}

\begin{proof}
This follows from the explicit construction of the classifying stack $\B\cG$ of $\cG$-torsors for the given Lie groupoid $\cG$. Now the associated Lie groupoid of the differentiable stack $\B\cG$ is Morita equivalent to the given Lie groupoid $\cG$ following Theorem \ref{MIst} above.
\end{proof}

\begin{example}
Consider the action groupoid $\cG=S^1\times S^3 \rightrightarrows S^3$ defined by the action of the circle $S^1$ on the $3$-sphere $S^3$ given by  $(v,w)\mapsto (z^3v,zw)$.

If we think of $S^3$ as the union of two solid tori, we have that the orbits of this action are circles of length $2\pi$ for points on the two cores $C_1$  and $C_2$ and circles of length $2\pi\sqrt{3\|v\|^2+\|w\|^2}$ elsewhere.   We construct a $\cG$-categorical covering $\{ \cU_1, \cU_2 \}$ by subgroupoids of $\cG$ given by considering the open sets  $U_1=S^3-C_2$ and $U_2=S^3-C_1$. We have that each inclusion $i_{\cU_i}\colon \cU_i \to \cG$ is groupoid homotopic to a generalized constant map with image in each respective core $C_i$. Moreover, this groupoid is not $\cG$-contractible and we have that $\cat (\cG)=2$.

Now for the differentiable stack $[S^3/S^1]$ associated to this Lie groupoid, which in fact describes the teardrop orbifold, we get from the last theorem that $\cat ([S^3/S^1])=2$.
\end{example}

More generally, given any quotient stack $[X/G]$ of a Lie group action on a smooth manifold, we see that the stacky LS-category of $[X/G]$ is equal to the groupoid LS-category of the action groupoid $G\times X\rightrightarrows X$, i.e. $\cat([X/G])=\cat(G\times X\rightrightarrows X)$. This stacky LS-category of a quotient stack in fact generalizes also the notion of equivariant LS-category $\cat_G(X)$ for Lie group actions on manifolds as introduced before by Marzantowicz \cite{Ma, Co0}. We also aim to study the relationship between stacky LS-category and equivariant topology in forthcoming work. 

Many of the particular examples and properties of LS-category for Lie groupoids as discussed in \cite{Co1}, especially concerning Lie group actions on smooth manifolds have interesting stacky analogs and will be discussed in detail in \cite{CN}. For example, orbifolds in general can naturally be seen as particular differentiable stacks associated to proper \'etale Lie groupoids and therefore give rise to a variety of interesting examples for LS-category of differentiable stacks and its relation with Morse theory, which we will also explore in detail in \cite{CN}.\\

\noindent {\it Acknowledgements:} 
The first and third author like to thank Sadok Kallel and the American University of Sharjah, UAE for financial support where part of this work was presented at the Second International Conference on Mathematics and Statistics. Both also like to thank the University of Leicester for additional support. 
The second author is supported in part by the Simons Foundation. Finally, the second and third authors also like to thank the Centro de Investigaci\'on en
Matem\'aticas (CIMAT) in Guanajuato for the kind hospitality and support while this project was pursued.

\end{document}